\documentclass[reqno,oneside]{amsart}

\usepackage[utf8]{inputenc}
\usepackage[super]{nth}
\usepackage{xcolor}
\usepackage{comment,fancyhdr,tikz,amsmath,amssymb,amsthm,mathtools,centernot,caption,float,graphicx,makecell,array,float,changepage,verbatim,textcomp,parallel,geometry,ragged2e,xcolor,lastpage,bbm,mathrsfs,pgfkeys,mathabx}

    \usepackage[shortlabels]{enumitem}

\usepackage[export]{adjustbox}
\usepackage[
    colorlinks=false,
    citebordercolor=green,
    linkbordercolor=red
    ]{hyperref}

    \usepackage[capitalise]{cleveref}
    \crefformat{equation}{(#2#1#3)}
    \crefformat{enumi}{(#2#1#3)}
    \crefname{section}{}{}
    \numberwithin{equation}{section}
    \newcommand{\crefdefpart}[2]{%
        \hyperref[#2]{\namecref{#1}~\labelcref*{#1}~\ref*{#2}}}
    
    \let\etoolboxforlistloop\forlistloop 
    \usepackage{autonum}
    \let\forlistloop\etoolboxforlistloop 
    \makeatletter
    \newcommand{\blx@noerroretextools}{}
    \makeatother
    
    \usepackage[
        backend=biber,
        style=alphabetic,
        maxbibnames=50,
        maxalphanames=50,
        maxcitenames=50
        ]{biblatex}
    \newcommand{\Z}{\mathbb{Z}}
    
    \newcommand{\R}{\mathbb{R}}
    \newcommand{\C}{\mathbb{C}}
    
    \def\sph^#1{\mathbb S^{#1}}
    \newcommand\p[1]{\left(#1\right)}
    \newcommand\abs[1]{\left|#1\right|}
    
    \newcommand{\floor}[1]{\left\lfloor#1\right\rfloor}
    
    \newcommand\jp[1]{\left\langle#1\right\rangle}
    \newcommand\jop[2]{\jp{#1,#2}}
    \newcommand\pl[1]{\left\|#1\right\|}
    \def\jopOld#1;{\left\langle{#1}\right\rangle}
    
    \newcommand\eps\varepsilon
    \newcommand\pji\varphi
    \newcommand\sig\varsigma

    \newcommand\beq{\begin{equation}}
    \newcommand\eeq{\end{equation}}
    \newcommand\inv{^{-1}}

    \def\XXint#1#2#3{{\setbox0=\hbox{$#1{#2#3}{\int}$ }
    \vcenter{\hbox{$#2#3$ }}\kern-.6\wd0}}

    \newcommand\length[1]{\text{length}\p{#1}}

    \newcommand\ef[2]{^{\frac{#1}{#2}}{}}
    \newcommand\ief[2]{^{-\frac{#1}{#2}}{}}
    \newcommand\ie[1]{^{-#1}{}}
    \newcommand\e[1]{^{#1}{}}

    \newtheoremstyle{thmst}
        {5pt}
        {5pt}
        {}
        {}
        {\bfseries}
        {.}
        {.5em}
        {}%
    \newtheoremstyle{rmkst}
        {5pt}
        {5pt}
        {}
        {}
        {\bfseries}
        {.}
        {.5em}
        {}%
    \newtheoremstyle{lmst}
        {5pt}
        {5pt}
        {}
        {}
        {\bfseries}
        {.}
        {5pt plus 1pt minus 1pt}
        {}%
    \newtheoremstyle{prpst}
        {5pt}
        {5pt}
        {}
        {}
        {\bfseries}
        {.}
        {.5em}
        {}%
    \newtheoremstyle{defst}
        {5pt}
        {5pt}
        {}
        {}
        {\bfseries}
        {.}
        {.5em}
        {}%
    \newtheoremstyle{wlst}
        {5pt}
        {5pt}
        {}
        {}
        {\bfseries}
        {.}
        {.5em}
        {}%
        
    \theoremstyle{thmst}
        \newtheorem{theorem}{Theorem}[section]
        \newtheorem{conjecture}[theorem]{Conjecture}
    \theoremstyle{rmkst}
        \newtheorem{remark}[theorem]{Remark}
    \theoremstyle{defst}
        
    \theoremstyle{defst}
        
    \theoremstyle{lmst}
        \newtheorem{lemma}[theorem]{Lemma}
    \theoremstyle{prpst}
        \newtheorem{proposition}[theorem]{Proposition}
    \theoremstyle{wlst}
        \newtheorem{whitelie}[theorem]{White Lie}

        \newlist{defenum}{enumerate}{1} 
        \setlist[defenum]{label=(\roman*),ref=\thedefinition\,(\roman*)}
        \crefname{defenumi}{Definition}{Definitions}
        
        \newlist{lemenum}{enumerate}{1} 
        \setlist[lemenum]{label=(\roman*),ref=\thelemma\,(\roman*)}
        \crefname{lemenumi}{Lemma}{Lemmas}
        
        \newlist{thmenum}{enumerate}{1} 
        \setlist[thmenum]{label=(\roman*),ref=\thetheorem\,(\roman*)}
        \crefname{thmenumi}{Theorem}{Theorems}
        
        \newlist{rmkenum}{enumerate}{1} 
        \setlist[rmkenum]{label=(\roman*),ref=\theremark\,(\roman*)}
        \crefname{rmkenumi}{Remark}{Remarks}
        
        \newlist{prpenum}{enumerate}{1} 
        \setlist[prpenum]{label=(\roman*),ref=\theproposition\,(\roman*)}
        \crefname{prpenumi}{Proposition}{Propositions}
        
        \newlist{axenum}{enumerate}{1} 
        \setlist[axenum]{label=(\roman*),ref=\theaxiom\,(\roman*)}
        \crefname{axenumi}{Axiom}{Axioms}
        
        \newlist{defcrit}{enumerate}{1} 
        \setlist[defcrit]{label=(\roman*),ref=\thedefinition\,(\roman*)}
        \crefname{defcriti}{Definition}{Definitions}
        
        \newlist{lemcrit}{enumerate}{1} 
        \setlist[lemcrit]{label=(\alph*),ref=\thelemma\,(\alph*)}
        \crefname{lemcriti}{Lemma}{Lemmas}
        
        \newlist{thmcrit}{enumerate}{1} 
        \setlist[thmcrit]{label=(\alph*),ref=\thetheorem\,(\alph*)}
        \crefname{thmcriti}{theorem}{theorems}
        
        \newlist{rmkcrit}{enumerate}{1} 
        \setlist[rmkcrit]{label=(\alph*),ref=\theremark\,(\alph*)}
        \crefname{rmkcriti}{remark}{remarks}
        
        \newlist{prpcrit}{enumerate}{1} 
        \setlist[prpcrit]{label=(\alph*),ref=\theproposition\,(\alph*)}
        \crefname{prpcriti}{proposition}{propositions}

    \newcommand\ext{\mathcal E}
    
    \newcommand\srf{\Sigma}

    \newcommand\ind{\text d}

        \makeatletter
        \let\ifnc\@ifnextchar
        \makeatother
        
            \def\<#1>{\jp{#1}}
            \def\-#1/{{}_{#1}}
            \makeatletter
                \newcommand\pl@write[3]{%
                    \left\|#1\right\|_{L^{#2}#3}%
                    }
                \def\pl #1_#2{%
                    \def\pl@arg@i{#1}%
                    \def\pl@arg@ii{#2}%
                    \def\pl@arg@iii{\alpha}%
                    \futurelet\next\pl@eval%
                    }
                \def\pl@eval{%
                    \ifx\next\bgroup%
                            \expandafter\pl@eval@iii%
                        \else%
                            \expandafter\pl@eval@ii%
                        \fi%
                    }
                \def\pl@eval@iii#1{%
                    \pl@write\pl@arg@i\pl@arg@ii{\p{#1}}%
                    }
                \def\pl@eval@ii{%
                    \pl@write\pl@arg@i\pl@arg@ii{}%
                    }
            \makeatother
            \def\MacroR R[#1.#2]{\begingroup\{#1,\cdots,#2\}\endgroup}
            \def\MacroF F#1/#2|{\begingroup\frac{#1}{#2}\endgroup}
            \def\MacroS S#1{\begingroup\{#1\}\endgroup}
            \def\MacroBeginB {}
            \def\MacroEndB {}
            \def\MacroL L#1{L^{#1}}
            \def\MacroPipe|#1|{\abs{#1}}
            \catcode`\@=\active
            \def@{
                \ifnc R{\MacroR}{
                \ifnc F{\MacroF}{
                \ifnc [{\MacroBeginB}{
                \ifnc ]{\MacroEndB}{
                \ifnc S{\MacroS}{
                \ifnc |{\MacroPipe}{
                \ifnc L{\MacroL}{
                \ifnc S{\MacroS}{
                \ifnc S{\MacroS}{
                \ifnc S{\MacroS}{
                \ifnc S{\MacroS}{
                \ifnc S{\MacroS}{}}}}}}}}}}}}}
\title{A Counterexample to the Mizohata-Takeuchi Conjecture}
\author{Hannah Cairo}
\date{February 2025}

\begin{document}
\nocite{*}

\begin{abstract}
    We derive a family of $L^p$ estimates of the X-Ray transform of positive measures in $\R^d$, which we use to construct a $\log R$-loss counterexample to the Mizohata-Takeuchi conjecture for every $C^2$ hypersurface in $\R^d$ that does not lie in a hyperplane. In particular, multilinear restriction estimates at the endpoint cannot be sharpened directly by the Mizohata-Takeuchi conjecture.
\end{abstract}

\maketitle

\section{Introduction}

    \def\curv{\texttt c_{\srf}}
    \def\ang{\texttt a_{\srf}}
    \def\td{\texttt d}
    \def\upart{\,\mathfrak U}
    \def\gdist{\texttt d}
    \def\lconst{\texttt l}
    \def\gauss{\gamma}
    \def\scale{{R\inv}}

Let $\srf\subset\R^d$ be a compact $C^2$ hypersurface with measure $\ind s$. The \textit{extension} operator $\mathcal E:L^1(\srf;\ind s)\to L^\infty(\R^d)$ is defined by
    \[[\ext f](x):=\int_{\srf}e^{-2\pi i\jop x\sig}f(\sig)\ind\sigma(\sig)\]

The Mizohata-Takeuchi conjecture can be stated as follows:

\begin{conjecture}[Mizohata-Takeuchi]\label{conjecture-mt}
    Let $\srf$ be any $C^2$ hypersurface in $\R^d$ with surface measure $\ind\sigma$. Let $f\in L^2(\srf,\ind\sigma)$ and let $w:\R^d\to\R_{\geq0}$ be a nonnegative weight. Then we have
        \[\int_{\R^d}\abs{\ext f(x)}^2w(x)\ind x\lesssim \pl{f}_2{\srf;\ind\sigma}^2\pl Xw_\infty
        \label{conj-mt}\]
    where $Xw$ denotes the X-Ray transform of $w$.
\end{conjecture}

The primary result of this paper is the following counterexample to the Mizohata-Takeuchi conjecture with $\log R$-loss:

\begin{theorem}[Counterexample]\label{thm-counterexample}
    For any $C^2$ hypersurface $\srf$ that is not a plane, there is some $f\in L^2(\srf;\ind\sigma)$ and nonnegative weight $w_R\R^d\to\R_{\geq0}$ so that the following holds.
        \[\int_{B_R(0)}\big|\ext f(x)\big|^2w(x)\ind x\gtrsim\log R\pl f_2{\srf;\ind\sigma}^2\sup_{\ell\subset\R^d\text{ a line}}\int_\ell w
        \label{eq-thm-counterexample}\]
\end{theorem}

\cref{conjecture-mt} originally arose in the study of well-posedness for dispersive PDE (\cite{takeuchi-necessary-conditions-74,takeuchi-cauchy-80,mizohata-cauchy-85}). Since then, the Mizohata-Takeuchi conjecture has taken on an important role in Fourier restriction theory for a few reasons, which we enumerate below.

\subsection{Multilinear Restriction Estimates}

In 2006, the following multilinear form of the restriction conjecture was formulated by \cite{bct-multilinear-06}:

\begin{conjecture}[Multilinear Restriction]\label{conjecture-multilinear-restriction}
Let $\{U_j:j\in[d]\}$ be a collection of $C^2$ hypersurfaces in $\R^d$, so that the normal to $U_j$ at any point is within $\frac1{100}$ of the $x_j$-axis. Let $\ext_j$ denote the corresponding extension operators. Then, for each $\eps>0, q \geq \frac{2d}{d-1}$ and $p^{\prime}\leq \frac{q(d-1)}d$,

\[
\left\|\prod_{j=1}^d \mathcal{E}_j g_j\right\|_{L^{q / d}(B(0, R))} \lesssim\prod_{j=1}^d\left\|g_j\right\|_{L^p\left(U_j\right)}\label{multilinear-restriction}
\]
for all $g_j \in L^p\left(U_j\right), 1 \leqslant j \leqslant d$, and all $R \geqslant 1$.
\end{conjecture}

Currently, \cref{conjecture-multilinear-restriction} has been proven away from the endpoint (\cite{tao-multirestriction-20}) and up to $R^\eps$ losses at the endpoint (\cite{bct-multilinear-06}).

In a recent paper, \cite{chv-weights-23} (see also \cite{chv-multilinear-18}) developed a general functional-analytic approach to dual formulations of multilinear estimates. In particular, they showed that The Mizohata-Takeuchi Conjecture implies \cref{conjecture-multilinear-restriction} without $R^\eps$ losses, using the endpoint multilinear Kakeya inequality of Guth (\cite{guth-endpoint-multi-kakeya-10}).

\cref{thm-counterexample} therefore shows that it is not possible to use this approach to prove \cref{conjecture-multilinear-restriction}.

\subsection{Stein's Conjecture}

In 1978, it was suggested by Stein (\cite{stein-conjecture-79}) that Kakeya or Nikodym maximal functions may control the behavior of Bochner-Riesz multipliers (see \cite{bcsv-stein-conjecture-06} and the references therein for the history of this approach).

In the 1990's, several papers (see \cite{crs-radial-mt-92,brv-radial-mt-97,cs-a-97,cs-b-97} and the references therein) were written on the subject, and the following conjecture has become known as Stein's conjecture:

\begin{conjecture}[Stein's Conjecture]\label{conjecture-stein}
    Under the hypotheses of \cref{conjecture-mt}, the following holds:
        \[\int_{\R^d}\abs{\ext f}^2w\ind x\lesssim\int_\srf|f(\sig)|^2\sup_{\ell\parallel N(\sig)}Xw(\ell)\ind\sigma(\sig)
        \label{eq-conj-stein}\]
\end{conjecture}

It is worthwhile for historical purposes to note that Stein originally posed several different forms of this inequality, and only recently has Stein's conjecture (in the context of the extension operator) come to refer to the inequality above (see \cite{bgno-phase-space-24}).

The reason why one might expect \cref{eq-conj-stein} to hold is as follows. One often expects $\ext f$ to be controlled in some sense by square functions, maybe of the form $\sum_{\Theta}|\hat{f|_\Theta\ind\sigma}|^2$ where $\Theta$ ranges over some collection of $R^\alpha$-caps for some $\alpha$. For weighted estimates on quantities of the form $\int_{\R^d}|\ext^2(x)|w(x)\ind x$, one expects the contribution from each term defining the square function to be concentrated along tubes by a parabolic rescaling argument. The norm $\pl f_2^2$ measures the concentration of the $L^2$ mass of $\ext f$ along these tubes in a way that is compatible with square functions, and so we are led to consider estimates of the form \cref{eq-conj-stein}.

Note that Stein's conjecture would directly imply the Mizohata-Takeuchi conjecture since

    \[\int_\srf|f(\sig)|^2\sup_{\ell\parallel N(\sig)}Xw(\ell)\ind\sigma(\sig)\leq\p{\int_\srf|f(\sig)|^2\ind\sigma(\sig)}\sup_{\sig\in\srf}\sup_{\ell\parallel N(\sig)}Xw(\ell)\ind\sigma(\sig)=\pl f_2{\srf;\ind\sigma}^2\pl Xw_\infty\]

Thus, \cref{thm-counterexample} implies that Stein's conjecture is false as stated in \cref{conjecture-stein}.

\subsection{The Kakeya Maximal Conjecture and the Restriction Conjecture}

The counterexample is also interesting because, assuming the Kakeya maximal conjecture, Stein's conjecture would imply the restriction conjecture (\cite{bgno-phase-space-24}).

Consider a Kakeya-type maximal operator $\mathcal M$ defined on suitably regular weights $w:\R^n\to\R_{\geq0}$ by
    \[\mathcal Mw(\omega)=\sup_{T\parallel \omega}\int_Tw,\qquad\omega\in\sph^{n-1}\]
where $T$ an infinitely long tube of width $1$. We denote the local maximal function by $\mathcal M_Rw=\mathcal M(w|_{B_R})$. Note that $\mathcal M_R$ is equivalent to the standard Kakeya maximal operator, upon a rescaling. As such, $\mathcal M_R$ is expected to satisfy the estimate
    \[\|\mathcal M_R\|_{L^n(\R^n)\to L^n(\sph^{n-1})}\lesssim_\eps R^\eps\label{kakeya-maximal-conj-eq}\]
which is equivalent to the Kakeya maximal conjecture.

To see how this is related to the restriction conjecture, take some $f\in L\ef{2n}{n-1}(\srf)$ and note that
    \[\pl \ext f_{\frac{2n}{n-1}}^2=\sup_{w\in L^n(\R^n)}\pl w_n\inv\int_{\R^d}\abs{\ext f}^2(x)w(x)dx\]
by H\"older's inequality. If Stein's conjecture were to hold, we have
    \[\pl{\ext f}_{\frac{2n}{n-1}}^2\lesssim\sup_{w\in L^n(\R^n)}\int_\srf|f(\xi)|^2\mathcal M(N_\srf(\xi))\ind\sigma(\xi)\]
From here, if \cref{kakeya-maximal-conj-eq} holds, then the above implies
\[\pl{\ext f}_{\frac{2n}{n-1}}{B_R}\lesssim_\eps R^\eps\pl f_{\frac{2n}{n-1}}{\srf}\]
if $\srf$ has nonvanishing curvature.

Note that the counterexample presented here only shows that Stein's conjecture is false up to a logarithmic factor; if Stein's conjecture were true locally, then the Kakeya maximal conjecture would still imply the restriction conjecture.

\subsection{$L^2$-wellposedness of first-order pertubations of the Schr\"odinger equation}

The historical value of the Mizohata-Takeuchi conjecture stems from its applications to the well-posedness of first-order pertubations of the Schr\"odinger equation. Around 1980, Takeuchi (\cite{takeuchi-necessary-conditions-74},\cite{takeuchi-cauchy-80}) was studying the Cauchy problem for operators of the form
    \[L(u)=i\partial_tu+\triangle u+\sum_{j=1}^nb_j\partial_ju+c(x)u=f
        \label{schrodinger}
        \]
The Cauchy problem for \cref{schrodinger} is $L^2$-wellposed if, given any $u_0\in L^2(\R^n)$ and $f\in C^0_t(L^2(\R^n))$, there is a unique solution $u(x,t)$ to \cref{schrodinger} for $|t|<T$ that satisfies the estimate
    \[\pl u(\cdot,t)_2\lesssim_T \pl {u_0}_2+\abs{\int_0^t\pl f(\cdot,s)_2\ind s}
        \label{wellposed}\]
Takeuchi (\cite{takeuchi-necessary-conditions-74,takeuchi-cauchy-80}) claimed that a sufficient condition for \cref{wellposed} to hold is the following:
    \[\text{Re}\p{\int_0^t\sum_{j=1}^nb_j(x+s\nu)\nu_j\ind s}\lesssim1
        \label{takeuchi-sufficient}\]
where the constant is independent of $(x,t,\nu)\in\R^n\times\R\times\mathbb S^{n-1}$. Later, Mizohata (\cite{mizohata-cauchy-85}) showed that Takeuchi's argument had an error.

It is well-known that the Cauchy problem for \cref{schrodinger} is closely related to weighted $L^2$ estimates of the form \cref{conj-mt}.
\subsection{The intrinsic value of the Mizohata-Takeuchi Conjecture}

Lastly, it is worthwhile to note that there is some intrinsic value in the Mizohata-Takeuchi conjecture. As noted in \cite{ciw-mt-24}, much of the development in Fourier restriction theory has centered around the $L^p$-$L^q$ mapping properties of $\ext$, and the Mizohata-Takeuchi and Stein conjectures are some of the very few attempts to understand the shape of the level sets of $\ext f$ for various $f:\srf\to\C$.
\subsection{Progress on the Mizohata-Takeuchi Conjecture}

In this section, we present a chronological overview of some important results in the direction of \cref{conjecture-mt}. One of the first influential results was established by Barcel\'o, Ruiz, and Vega (\cite{brv-radial-mt-97}) which demonstrated the radial case of the Mizohata-Takeuchi conjecture (see also \cite{crs-radial-mt-92}). This was achieved by directly estimating Bessel functions and the extension operator on spherical harmonics.

In 2009, Carbery (\cite{carbery-tubes-09}) studied the nature of sets $U\subset\R^d$ where $X(1_U)$ is small. In general, the study of tube-occupancy of sets is difficult, independently of weighted Fourier extension estimates (see \cite{rd-nset-24}).

More recently, Shayya (\cite{shayya-plane-mt-23}) has studied the extent to which Mizohata-Takeuchi-type estimates can be derived from decay properties of $\widehat{\ind\sigma}$ in the plane. Another important recent development was by Ortiz (\cite{ortiz-cone-mt-23}), in which the Mizohata-Takeuchi conjecture was proven with an $R^{\frac14+\eps}$-loss for the cone in $3$ dimensions.

Probably the most influential recent development in the direction of the Mizohata-Takeuchi conjecture is a result by Carbery, Iliopoulou, and Wang (\cite{ciw-mt-24}), in which several special cases (for the weight) were proven, as well as the general conjecture with $R^{\frac{n-1}{n+1}+\eps}$-loss. The approach was based on the refined decoupling estimates of Guth, Iosevich, Ou, and Wang (\cite{giow-refined-decoupling-20}).

The $R\ef{n-1}{n+1}$-losses are especially relevant when placed in the context of a recent talk given by Guth (\cite{guth-barrier-22}), in which he demonstrates that it is not possible to prove the Mizohata-Takeuchi conjecture without $R\ef{n-1}{n+1}$ loss using only certain decoupling axioms.

More recently,  Bennett, Gutierrez, Nakamura, and Oliveira (\cite{bgno-phase-space-24}) have explored connections to time-frequency analysis and have proven a related Mizohata-Takeuchi-type inequality involving Sobolev norms.

\subsection{Reformulations}

In light of \cref{thm-counterexample}, the following local version may be a plausible reformulation of the Mizohata-Takeuchi conjecture.

\begin{conjecture}[Local Mizohata-Takeuchi]\label{conjecture-refine-mt}
    Under the same hypotheses as \cref{conjecture-mt}, we have
        \[\int_{B_R}\abs{\ext f(x)}^2w(x)\ind x\lesssim_\eps R^\eps\pl f_2{\srf;\ind\sigma}^2\sup_{\ell\subset\R^d\text{ a line}}\int_\ell w
        \label{conj-refine-mt}\]
\end{conjecture}

It is unclear whether one should expect \cref{conjecture-refine-mt} to hold, or to expect an $R\ef{n-1}{n+1}$-loss counterexample in the spirit of Guth's argument (\cite{guth-barrier-22}).

\subsection{Outline of this paper}

We begin in Section \cref{xray-estimates} by proving $L^2$-based estimates on the X-Ray transform for positive measures. We then discuss how this can be used to formulate a weaker version of the Mizohata-Takeuchi conjecture on the Fourier side in Section \cref{sec-proof-thm-counter}. The construction relies on an incidence geometry Lemma, which is proven separately in Section \cref{sec-incidence}.


\subsection{Acknowledgments}

The author would like to express her gratitude to her advisor, Ruixiang Zhang, for his generous introduction to the beautiful theory of Fourier restriction, as well as for pointing out the Mizohata-Takeuchi conjecture and reviewing earlier versions of this paper.

\section{X-Ray transform estimates for positive measures}\label{xray-estimates}

Let $\nu\in\mathbb S^{d-1}$. For any weight $w:\R^n\to[0,\infty)$, we define $X_\nu w(z)=\int_{\R\nu+z}w$ for any $z\in\nu^\perp$. For any function $\mu:\R^n\to\C$ so that $\mu$ restricts to an $L^p$ function on every plane perpendicular to $\nu$, we define the restriction map $P_\nu\mu:\R\to L^p(\nu^\perp)$ by $P_\nu\mu(\lambda)(\omega)=\mu(\lambda\nu+\omega)$.

In this section, we prove

\begin{theorem}
    Let $h:\R^n\to\C$ be so that $|\hat h|^2=w$, for a nonnegative weight $w:\R^n\to[0,\infty)$. Then we have the estimate
        \[\pl {X_\nu}w_p\leq\pl \big.{P_\nu h}_{2}{L^q(\nu^\perp)}^2
            \label{hy-xray}\]
    for $p\in[1,\infty]$ and $q=\frac{2p}{2p-1}$. In particular, when $p=\infty$ we have the estimate
        \[\pl {X_\nu}w_\infty\leq\pl \big.{P_\nu h}_{2}{L^1(\nu^\perp)}^2
            \label{hy-xray-infty}\]
     and in this case, the estimate is sharp when $h\geq0$ everywhere.
\end{theorem}
\begin{proof}
    The proof is relatively straightforward, once we digest the meaning of the norm appearing on the right of \cref{hy-xray}. We write
        \[\pl {P_\nu} h(\lambda)_q{\nu^\perp}=\p{\int_{\nu^\perp}\abs{h(\lambda\nu+\omega)}^q}\ef1q\]
    Thus,
        \[\pl \big.{P_\nu h}_{2}{L^q(\nu^\perp)}=\p{\int_\R\p{\int_{\nu^\perp}\abs{h(\lambda\nu+\omega)}^q}\ef2q}\ef12\]
    We note that $X_\nu w=\widehat{P_\nu{\;\check~\!\!\!\!w}(0)}$ by the projection-slice theorem. We note that
        \[P_\nu\check w(0)=P_\nu(h*\tilde{\hat h})(0)=\int_\R P_\nu h(\lambda)*\widetilde{\overline{P_\nu h(\lambda)}}\]
    where $\tilde\mu$ denotes $\mu\circ(x\mapsto-x)$ for any function $\mu$. Therefore, we have
        \[X_\nu w(z)=\int_\R \abs{\widehat{P_\nu h(\lambda)}(z)}^2\ind\lambda\]
    Minkowski's inequality gives
        \[\pl{X_\nu w}_p{\nu^\perp}\leq\int_\R\pl{\widehat P_\nu h(\lambda)}_{2p}{\nu^\perp}^2\]
    Applying Hausdorff-Young yields \cref{hy-xray}. Note that, if $h\geq0$ everywhere and $p=\infty$, then both Minkowski's inequality and Hausdorff-Young are sharp.
\end{proof}

\section{The proof of Theorem \ref{thm-counterexample}}\label{sec-proof-thm-counter}

This section is divided into four parts:
    \begin{enumerate}
        \item[\textbf{\cref{subsec-prelim}.}] We discuss some preliminary estimates. This part is primarily devoted to posing similar problems of Mizohata-Takeuchi-type on the Fourier side.
        \item[\textbf{\cref{subsec-wh}.}] We present a ``white lie'' proof of \cref{thm-counterexample} to help the reader understand the important points.
        \item[\textbf{\cref{subsec-lemma}.}] We explain briefly an incidence lemma, which we prove in section \cref{sec-incidence}.
        \item[\textbf{\cref{subsec-proof}.}] We present a rigorous proof of \cref{thm-counterexample}.
    \end{enumerate}
\subsection{Some preliminaries}\label{subsec-prelim}

Let $w:B_R\to[0,\infty)$ be a nonnegative weight. As \cite{ciw-mt-24} notes, we are free to assume that $w$ is locally constant at scale $1$, and so $w$ satisfies the necessary regularity conditions discussed in the previous section. We are interested in the quantity $E_w(f,g):L^2(\srf;\ind\sigma)\times L^2(\srf;\ind\sigma)\to\C$ given below.
    \[E_w(f,g):=\int_{\R^d}\ext f(x)\overline{\ext g}(x)w(x)\ind x=\jop{w\ef12\ext f\ind\sigma}{w\ef12\ext g\ind\sigma}
        \label{def-E-w}\]
We should think of $E_w$ as a quadratic form on the Hilbert space $L^2(\srf;\ind s)$. The Mizohata-Takeuchi conjecture seeks an upper bound on the largest eigenvalue of $E_w$ in terms of $\pl Xw_\infty$.

As in the previous section, let $h:\R^n\to\C$ so that $\abs{\hat h(x)}^2=w(x)$. We write
    \[E_w(f,g)=\jop{h*f\ind\sigma}{h*g\ind\sigma}
        \label{eq-plancherel-reformulation}\]
In particular, we note $E_w(f,f)=\pl h*f\ind\sigma_2^2$. Therefore, to prove \cref{thm-counterexample}, it suffices to prove the following.

\begin{proposition}\label{prop-reform}
    For any $C^2$ hypersurface $\srf\subset\R^d$, there exists for each $R\geq1$, a function $f\in L^2(\srf;\ind\sigma)$ and a function $h:\R^d\to\C$ with $\hat h\subset B_R$, so that the following holds.
        \[      \pl h*f\ind\sigma_2^2
            \gtrsim
                \log R
                \pl f_2{\srf;\ind\sigma}^2
                \pl\big.{P_\nu h}_{2}{L^1(\nu^\perp)}^2
            \]
    for every $\nu\in\mathbb S^{d-1}$
\end{proposition}

In light of \cref{prop-reform}, we are led to consider partial progress towards \cref{conjecture-refine-mt} in the following form.

\begin{conjecture}
    For any $C^2$ hypersurface $\srf\subset\R^d$ with surface measure $\ind\sigma$ and any functions $f\in L^2(\srf;\ind\sigma)$ and $h:\R^d\to\C$ with $\hat h\subset B_R$, the following holds
        \[      \pl h*f\ind\sigma_2^2
            \lesssim_\eps
                R^\eps
                \pl f_2{\srf;\ind\sigma}^2
                \pl\big.{P_\nu h}_{2}{L^1(\nu^\perp)}^2
            \]
    for every $\nu\in\mathbb S^{d-1}$
\end{conjecture}

To help the reader understand the construction of the counterexample, we present first a ``white lie'' version.

\subsection{A white lie construction}\label{subsec-wh}
Throughout this section, we use $A\approx B$ to mean that $A$ and $B$ are \textit{essentially} equal, in the sense that $B$ is a suitable approximation of $A$ in some sense that we will formalize in a later section. The reader may feel free to assume $A=B$ when verifying estimates, even though $A=B$ is generally false. The point of this notation is to avoid excess detail at the expense of rigor.

Let $\xi_1,\cdots,\xi_N\subset\srf$ be a $R\inv$-separated collection of $N\sim\log R$ points in $\srf$. We will choose the $\xi_i$ later on. Let us set $S_i=\pl{1_{B_{R\inv}(\xi_i)\cap\srf}}_1{\srf;\ind\sigma}\inv1_{B_R(\xi_i)\cap\srf}\ind\sigma$ and
    \[f=\sum_{i=1}^NS_i.
    \label{def-f-white-lie}\]
We record the following for future use:
    \[R^{-d+1}\pl f_2{\srf;\ind\sigma}^2\sim\pl f_1{\srf;\ind\sigma}\sim N
    \label{f-norms-white-lie}\]
The key step is the construction of the following lattice:
    \[Q=\bigg\{\vec c\cdot\vec\xi:\vec c\in\{0,1\}^N,c_1+\cdots+c_N=\floor{\frac N2}\bigg\}\]
Here, $\vec c\cdot\vec \xi$ denotes $c_1\xi_1+\cdots+c_N\xi_N$ by abuse of notation. We now define
    \[h=R^d\sum_{q\in Q}1_{B_{R\inv}(q)}
    \label{def-h-white-lie}\]
Note that $|Q|={N\choose{\floor{\frac N2}}}\sim N\ief122^N$ by Stirling's approximation for the factorial.

The choice of scale $R\inv$ reflects the locally constant principle:
\begin{whitelie}\label{wh-local-const}
    We have $\text{supp}(\hat h)\subset B(R)$.
\end{whitelie}

We would like a lower bound on $\pl f\ind\sigma*h_2^2$. The operation $h\mapsto f\ind\sigma*h$ sends each ball in \cref{def-h-white-lie} to the sum of the translations of that ball by each of $\xi_1,\cdots,\xi_N$. We will need an estimate of the following type.
\begin{whitelie}\label{wh-convo}
    We have $S(\xi_i)*1_{B(R\inv)}\approx R^{d}1_{B(\xi_i,R\inv)}$
\end{whitelie}
With this, we are free to write
    \[h*f\approx\sum_{i=1}^N\sum_{q\in Q}R^d1_{B(\xi_i+q,R\inv)}
    \label{eq-h-f-white-lie}\]
Now, note that $\geq\frac12$ of the possible values of $\xi_i+q$ satisfy $c_i=0$, if $q=\vec c\cdot\vec\xi$. Therefore, at least $\frac12$ of the values of $\xi_i+q$ will lie in the lattice
    \[Q'=\bigg\{\vec c\cdot\vec\xi:\vec c\in\{0,1\}^N,c_1+\cdots+c_N=\floor{\frac N2}+1\bigg\},\]
which is a set of size at most ${N\choose\floor{\frac N2}+1}\sim|Q|$. Since there are a total of $N|Q|$ balls in \cref{eq-h-f-white-lie}, where each ball has an $L^2$ norm of $R\ef d2$, we have
    \[\pl h*f_2^2\gtrsim N^2|Q|R^d
    \label{est-lower-hf-white-lie}\]
We would now like to estimate $\pl\big.{P_\nu h}_{2}{L^1(\nu^\perp)}^2$. We recall
    \[\pl\big.{P_\nu h}_{2}{L^1(\nu^\perp)}^2:=\int_\R\pl{P_\nu h(\lambda)}_1^2\ind\lambda\]
Let $K_\nu(\lambda)$ denote the number of balls in the definition \cref{def-h-white-lie} of $h$ that pass through the plane $\lambda\nu+\nu^\perp$, so that $\pl{P_\nu h(\lambda)}_1\lesssim K_\nu(\lambda)R$. We deduce that $\pl{P_\nu h(\lambda)}_1^2\lesssim R^2\pl{K_\nu}_2^2$. We use $\pl{K_\nu}_1\sim R\inv 2^N$ to obtain
    \[\pl{P_\nu h(\lambda)}_1^2\lesssim \pl{K_\nu}_1\pl{K_\nu}_\infty\sim R2^N\pl{K_\nu}_\infty
    \label{length-estimate-white-lie}\]
Combining \cref{f-norms-white-lie}, \cref{est-lower-hf-white-lie}, and \cref{length-estimate-white-lie}, we obtain
    \[\pl h*f\ind\sigma_2^2\pl{K_\nu}_\infty\gtrsim N\pl f_2^2\pl{P_\nu h(\lambda)}_1^2
    \label{counterexample-white-lie}\]
We will show that $\pl{K_\nu}_\infty\lesssim1$ for every $\nu$, for some choice of $\{\xi_1,\cdots,\xi_N\}$ in \cref{main-lemma} below, i.e. no plane passes through more than $\sim 1$ balls in the definition \cref{def-h-white-lie} of $h$. The ``white-lie'' proof is now complete, contingent upon the following lemma.

\subsection{An incidence geometry lemma}\label{subsec-lemma}
    \begin{lemma}\label{main-lemma}
    For any $R>1$ sufficiently large, there exists
        \begin{enumerate}
            \item[$\circ$] A point $\xi_0\in\srf$
            \item[$\circ$] A set of $N\sim\log R$ points $\{\xi_1,\cdots,\xi_N\}\subset\srf$
        \end{enumerate}
        so that the following conditions are met.
        \begin{lemenum}
            \item\label{plane-infty-main-lemma}
                No plane passes through more than $2^{d-1}$ balls in the collection of balls of radius $R\inv$ around the points $c_1(\xi_1-\xi_0)+\cdots+c_N(\xi_1-\xi_0)$ for $\vec c\in\{0,1\}^N$.
            \item\label{xi-separated-main-lemma}
                We have $\abs{\xi_m-\xi_n}>\scale$ for every $(m,n)\in[N]^2$.
        \end{lemenum}
\end{lemma}

We note that $Q$ is a translate of $\{c_1(\xi_1-\xi_0)+\cdots+c_N(\xi_1-\xi_0):\vec c\in\{0,1\}^N,\sum_{i=1}^Nc_i=\floor{\frac N2}\}$ and so the point $\xi_0$ does not affect our estimates.

The proof of this Lemma is rather involved and may be of independent interest. We postpone the proof to Section \cref{sec-incidence}.

\subsection{A rigorous proof of Proposition \ref{prop-reform}}\label{subsec-proof}
We now present a rigorous proof of \cref{prop-reform} by the means described in Section \cref{subsec-wh}.

Let $\xi_1,\cdots,\xi_N$ with $N\sim\log R$ be a collection of points that satisfy the conclusion of \cref{main-lemma}. Let $\eta$ be so that $\hat\eta$ is a smooth cutoff function with $1_{B(C\inv)}\leq\hat\eta\leq 1_{B(C)}$ for some constant $C>1$ to be determined later, so that $\hat\eta$ is supported in a ball of radius $\sim 1$; let $\eta_R(\xi)=R^d\eta(R\xi)$ be a rapidly decaying kernel at scale $R\inv$. For convenience, we assume that $\eta$ is radially symmetric. We set
    \[h_0=\sum_{p\in Q}\delta_p,
        \label{def-h0}\]
and choose $h=h_0*\eta_R$ so that $\hat h\subset B_{CR}$. Note that $\pl{h_0}_1\sim\pl h_1\sim|Q|$ and $\pl h_2^2\sim|Q|R^d$.

We set $f_i=\sum_iS_i$ as in the previous section. We would like to show the following estimate.
    \[\pl f\ind\sigma*h_2^2=\pl f\ind\sigma*{h_0}*{\eta_R}_2^2\gtrsim N^2|Q|R^d
        \label{est-fh}\]
We have
    \begin{align}
            \pl f*{h_0}*{\eta_R}_2^2
        =&
            \int_{\srf\times\srf}\sum_{i,j=1}^N\sum_{q,q'\in Q}S_i(\sig)S_j(\sig')[\eta_R*\eta_R](\sig-\sig'+q-q')\ind\sigma(\sig)\ind\sigma(\sig')\label{tail-bound}
    \end{align}
Note that $\eta_R*\eta_R=\widecheck{(\widehat{\quad}\!\!\!\!\!\!\eta_R)^2}$ is another rapidly decaying kernel at scale $R\inv$. We morally have $\eta_R*\eta_R\lesssim1_{B_{2R\inv}}$ but the honest kernel has rapidly decaying oscillatory tails. If we set 
    \[H=\sum_{i,j=1}^N\sum_{q,q'\in Q}\delta_{\xi_i-\xi_j+q-q'}*T_{ij}\]
where $T_{ij}(\xi_j-\xi_i+\zeta)=S_i*\tilde S_j(\zeta)$, then we can rewrite \cref{tail-bound} as
    \[
        \pl f*{h_0}*{\eta_R}_2^2=\langle H,\eta_R*\eta_R\rangle,
    \]
Note that if $q$, $\xi_i$, and $\xi_j$ are randomly chosen, then there is a $\geq\frac14$ probability that $q+\xi_i-\xi_j$ will lie in $Q$. Therefore, we deduce 
    \[\int_{B_{R\inv}(0)}H\gtrsim N^2|Q|\label{big-mass-0}\]

Now, note that we can deduce that that points in $Q$ are separated by at least $R\inv$ from \cref{main-lemma}, in the sense that every $R\inv$-ball contains $\lesssim1$ of the points (in fact, their $\nu$-projections are separated). Since $\pl{T_{ij}}_1\sim 1$, $T_{ij}$ is nonnegative, and $T_{ij}$ is supported in $B_{2R\inv}(0)$, we deduce that
    \[\int_{B_{R\inv}(\zeta)}H\lesssim N^2|Q|.\label{small-mass-zeta}\]
for any $\zeta\in\R^d$.

Combining \cref{big-mass-0} and \cref{small-mass-zeta}, we conclude that \cref{est-fh} holds for a suitable choice of $\eta$ so that $\eta*\eta\geq B_1(0)$ and $\int_{\R^d}(\eta*\eta)_-<c$ for some sufficiently small $c$ (here, $\eta*\eta=(\eta*\eta)_+-(\eta*\eta)_-$ is a difference of nonnegative functions).

To estimate $\pl\big.{P_\nu h}_{2}{L^1(\nu^\perp)}^2$, we note that the method from Section \cref{subsec-wh} must be modified slightly, since $\eta$ has a rapidly decaying tail, and therefore we expect some interference between the different ``layers'' of $Q$. More concretely, we write
    \[\pl{P_\nu h(\lambda)}_1\lesssim \int_\R K_\nu(\lambda')\mu_R(\lambda-\lambda')\ind\lambda'=K_\nu*\mu_R(\lambda)\]
where $K_\nu$ is defined as in Section \cref{subsec-wh} and $\mu_R(\lambda)=R\mu(R\lambda)$ is some rapidly decaying smooth kernel at scale $R\inv$. We write
    \begin{align}
        \pl{P_\nu h}_2{L^1(\nu^\perp)}^2
    \lesssim&
        \pl{RK_\nu*\mu_R}_2^2
    \\\leq&
        R^2\pl{K_\nu}_1\pl{K_\nu}_\infty\pl{\mu_R}_1^2
    \\\lesssim&
        R^2\pl{K_\nu}_1\pl{K_\nu}_\infty
    \\\lesssim&
        R|Q|
    \end{align}
This completes the proof.

\section{The proof of Lemma \ref{main-lemma}}\label{sec-incidence}
Let $\srf\subset\R^d$ be a compact $C^2$ hypersurface that is not a subset of a plane. We will see in section \cref{subsec-prove-lemma} that \cref{main-lemma} will follow easily from the following Lemma.

\def\jop#1\nu{\pi_\nu\p{#1}}
\def\jopOld #1,\nu;{\pi_\nu\p{#1}}
\def\jopOldWeird ,#1;{\pi_\nu\p{#1}}
\begin{lemma}\label{lemma-pigeonholing}
    For any $R>1$ sufficiently large, there exists
    \begin{enumerate}
        \item[$\circ$] A point $\xi_0\in\srf$
        \item[$\circ$] A set of $N\sim\log R$ points $\{\xi_1,\cdots,\xi_N\}\subset\srf$
    \end{enumerate}
    so that the following conditions are met.
    \begin{lemenum}
        \item\label{pigeonholing-condition}
            For any direction $\nu\in\mathbb S^{d-1}$, the $\nu$-projections of the differences $\xi_m-\xi_1$ lie in nearly distinct dyadic intervals; that is, if $\mathfrak u_n$ denotes the part of the partition $\upart^1(1)$ that contains $\log_2(\abs{\jopOldWeird,\xi_n-\xi_0;})$, then we have
                \[\mathfrak u_m=\mathfrak u_n\text{ and }m>n\implies n\in S_\nu
                \label{pigeons!}\]
            where $S_\nu$ is called the set of bad values of $n$. The size of $S$ is at most $d-1$.
        \item\label{xi-separated}
            We have $\abs{\xi_m-\xi_n}>\scale$ for every $(m,n)\in\{1,N\}^2$.
    \end{lemenum}
\end{lemma}
\begin{remark}\label{rmk-lem-pigeon}
    It is much easier to see that the lemma is true if we are allowed to choose different $\xi_n$ for different $\nu$; since $\srf$ does not lie in a hyperplane, it cannot project to a point, so $\pi_\nu(\srf)$ must contain some line segment from which we can choose $\xi_n$ dyadically.
\end{remark}
Before proving \cref{lemma-pigeonholing}, we explain the main ideas. First, the condition that $\srf$ is a hypersurface is not necessary as long as we assume that it does not project to a point and is $C^2$. We will also see that higher order corrections do not play an important role in the proof, so we might as well assume that $\srf$ is smooth. In this case, the desirable region of $\srf$ is best approximated by the \textit{moment curve} $\mathcal M_d$.
    \[\mathcal M_d(t)=(t,t^2,\cdots,t^d)
    \label{def-moment-curve}\]
That is, we should be able to construct $\xi_n\approx\mathcal M_d(t_n)$ for a nice choice of $t_n$ and most of the $\mathcal M_d(t_n)$ will nearly lie in $\srf$. Note that if $\nu=(\nu_1,\cdots,\nu_d)$ was close to $(1,\cdots,0)$, say $|\nu_1|>c>0$, then the projection $\jop{\mathcal M_d(t)}\nu$ would look like a line near $t=0$. Motivated by this, we might decide to set $t_n=2\ie n$ dyadically. This also works well if, e.g., $\nu=(0,1,\cdots,0)$ because then $\jop{\mathcal M_d(t_n)}\nu\sim2\ie{2n}$. It is remarkably harder to prove that the estimates are uniform over all possible choice of lines.

It is, in general, difficult to work with neighborhoods of the projections $\jop{\mathcal M_d(t_n)}\nu$ and it is usually preferable to work with neighborhoods of $t_n$ prior to projection. We might try setting $B_n=B(\mathcal M_d(t_n),2|\mathcal M_d(t_{n-1})|)$ so that $0\notin\jop{B_n}\nu$ implies that $\jop{\mathcal M_d(t_n)}\nu$ lies in a dyadic interval higher than $\mathcal M_d(|t_{n-1}|)$, and therefore higher than $\jop{\mathcal M_d(t_{n-1})}\nu$. The advantage here is that the problem reduces to showing something about how many balls lie on a hyperplane $\nu^\perp$ through the origin, which seems much more tractable.

Unfortunately, this approach fails. To see why, let us choose $t_n=c\ie n$ for some $c>1$, and set $d=2$ for definiteness, so that the $B_n$ are balls of radius $\sim c\inv|t_n|=c\ie1\ie n$. If we choose $\nu=(0,1)$, then the $B_n$ project to segments of length $\sim c\ie1\ie n$ a distance $\sim c\ie{2n}$ away from $0$. As $n$ gets large, these segments will overlap a lot because the lengths do not decay fast enough.

In light of the above setback, we might be temped to choose the $B_n$ to be rectangles instead. If we set $U_n=t_n+[-2c\e{-1-n},2c\e{-1-n}]\times[-2c\e{-2-2n},2c\e{-2-2n}]$, then we might hope that $\jop{U_n}\nu$ contains the origin when $\jop{t_n}\nu$ does not lie in a dyadic interval higher than $\jop{t_{n-k}}\nu$. That is, if $\frac{\jop{t_n}\nu}{\jop{t_{n-k}}\nu}\in[-2,2]$ for some $k>1$, we would like to see if $0\in\jop{U_n}\nu$. If $\jop{t_n}\nu=\alpha\jop{t_{n-k}}\nu$ for some $\alpha\in[-2,2]$, then we would like to show that $t_n-\alpha t_{n-k}\in U_n$ because we know that $\jop{t_n-\alpha t_{n-k}}\nu=0$. As it turns out, it is not hard to verify that $t_n-\alpha t_{n-k}\in U_n$ from the definition of $U_n$.

When $d>2$, the main idea of the proof is to use a projection $\varphi:\R^d\to\R^{d-1}$ to induct on $d$. We define $\varphi$ below.
    \[\varphi(x_1,\cdots,x_d)=\frac{(x_2,\cdots,x_d)}{x_1}
    \label{def-varphi}\]
We will see that $\varphi$ is useful because it reduces the dimension by $1$ and it sends hyperplanes in $\R^d$ to (affine) hyperplanes in $\R^{d-1}$. This is the point where the choice \cref{def-moment-curve} of the moment curve is essential because $\varphi\circ\mathcal M_d=\mathcal M_{d-1}$. After suitably defining the rectangles $U^d_n$ in dimension $d>2$, we will note that $0\in\jop{U^d_n}\nu$ is equivalent to the statement that the hyperplane $\nu^\perp$ that contains $0$ also passes through some point of $U^d_n$. This will then imply that $\varphi(\nu^\perp)$ meets some point of $\varphi(U^d_n)$. However, $\varphi(\nu^\perp)$ is just a hyperplane in $\R^{d-1}$ and we will see that $\varphi(U^d_n)$ is very close to $U_{n,d-1}$ and so the problem reduces to a $d-1$-dimensional problem.

There is one more complication which arises, namely that after applying $\varphi$, we are dealing with hyperplanes in general position. Before applying $\varphi$, we were only dealing with hyperplanes through the origin.\footnote{It is notable that this is the only reason that the number of ``bad'' $n$, i.e. the constant in \cref{pigeons!}, is dependent on $d$.} To remedy this, consider fixing a scale $\texttt d_n\sim|t_n|$ for some $n$. Note that either $\varphi(\nu^\perp)$ will be very far away from the origin, it will be very close to the origin, or it will be a moderate distance away from the origin. If $\varphi(\nu^\perp)$ is very far from the origin at scale $\texttt d_n$, then there is no risk of $\varphi(\nu^\perp)$ passing through $\varphi(U^d_n)$. On the other hand, if $\varphi(\nu^\perp)$ is very close to the origin at scale $\texttt d_n$, then we might as well assume that it meets the origin, in which case we have reduced the problem completely to the $n-1$-dimensional case. For the third case, namely that the distance from $\varphi(\nu^\perp)$ to the origin is the same scale as $\texttt d_n$, we are truly without hope. Luckily, it is not possible for very many $t_n$ to meet this fate since there are not many scales $\texttt d_n$ that are close to the distant from $\varphi(\nu^\perp)$ to the origin.

In view of the principles above, we are ready to formalize the proof.

\setcounter{subsection}2
\subsection{The proof of Lemma \ref{lemma-pigeonholing}}

We divide the proof into two lemmas. \cref{lemma-moment-curve} is the main argument for the moment curve and \cref{lemma-srf-sufficient} proves that \cref{lemma-pigeonholing} for an arbitrary curved $C^2$ hypersurface $\srf$ reduces to the moment-curve case.

Recall the definition of the moment curve.
\[\mathcal M_d(t)=(t,t^2,\cdots,t^d)
    \tag*{\cref{def-moment-curve}}\]
For any $c>1$, we set $t_{c,n}=c\ie n$ and $x_{c,n}=\mathcal M_d(t_{c,n})$. Throughout the proof, we will use the rescaling symmetry of the moment curve.
    \[\mathcal M_d(c\inv t)=L_c\circ\mathcal M_d(t)
    \label{moment-rescaling}\]
Here, $L_c(x)$ is the entry-wise multiplication of $x$ by $\mathcal M_d(c\inv)$. Notably, we have
    \[
        \label{symmetry-x}
            L_c(x_{c,n})\;=\;L_c(x_{c,n+1})
    \]
We also recall the definition of $\varphi:\R^d$.
    \[\varphi(x_1,\cdots,x_d)=\frac{(x_2,\cdots,x_d)}{x_1}
    \tag*{\cref{def-varphi}}\]
At the core of the argument is the fact that $\varphi$ and $L_c$ commute:
    \[\varphi\circ L_c=L_c\circ\varphi
    \label{lc-varphi-commute}\]
Here, we written $L_c$ for the corresponding scaling in $\R^{d-1}$ by abuse of notation. In fact, by a similar abuse of notation, all of $L_c$, $\mathcal M$, and $\varphi$ mutually commute.

Given any vector $v$, let us denote by $Q(v)$ the axially-oriented rectangular box centered at the origin, whose dimensions are encoded in the entries of the vector $2v$. We should imagine $v$ as a corner of the box. We now define the boxes $U^d_{b,c,n}$, where $b>1$ will be fixed later.
    \[U^d_{b,c,n}=x_{c,n}+Q(\mathcal M_d(bt_{c,n+1}))
    \label{def-ubcn}\]
In the notes following \cref{rmk-lem-pigeon}, we took $b=2$. To digest \cref{def-ubcn}, recall that we are looking for the smallest box so that translations $x_{c,n}+\alpha x_{c,k}$ remain in the box, for suitably defined $\alpha$.

Since $Q$ commutes with axial rescaling, the boxes $U^d_{b,c,n}$ are related by $L_c$.
    \[
        \label{symmetry-U}
            L_c(U^d_{b,c,n})\;=\;U^d_{b,c,n+1}
    \]
We are now ready to state and prove the following lemma.

\begin{lemma}\label{lemma-moment-curve}
Let $d$ be arbitrary. For $c\gg b\gg 1$ sufficiently large, we have the following.
    \begin{lemenum}
        \item For any direction $\nu\in\mathbb S^{d-1}$ and any $(n,k)\in\Z^2$ with $n<k$, the projections $\jopOld U^d_{b,c,n},\nu;$ and $\jopOld U^d_{b,c,k},\nu;$ do not overlap unless $n\in S_\nu$. Here, $S_\nu$ is called the set of bad $n$ and $|S_\nu|\leq d-1$.
        \item In particular, the values of $\log_c\jopOld c_{c,n},\nu;$ are distinct, except for a set of as most $d-1$ choices of $n$.
    \end{lemenum}
\end{lemma}
\begin{proof}
The proof is by induction on $d$. Note that the $d=1$ case is trivial.

The first step is to note that $\jopOld U^d_{b,c,n},\nu;$ and $\jopOld U^d_{b,c,k},\nu;$ overlap iff $\nu^\perp$ passes through a point of $U^d_{b,c,n}-U^d_{b,c,k}$. Note that all of the boxes $Q(\mathcal M_d(bt_{c,k+1}))$ in the definition \cref{def-ubcn} of $U^d_{b,c,n}$ will fit inside $Q(\mathcal M_d(bt_{c,n+1}))$. The points $x_{c,k}$ also fit inside this box, so we deduce the following.
    \[U^d_{b,c,n}-U^d_{b,c,k}\subset U^d_{b_1,c,n}
    \label{set-difference-bound}\]
Here, $b_1\sim b$ is some other scale. Therefore, it suffices to construct a set $S_\nu$ of bad points so that for $n\notin S_\nu$, the hyperplane $\nu^\perp$ will not meet $U^d_{b_1,c,n}$. This is equivalent to proving that $\varphi(\nu^\perp)$ does not meet $\varphi(U^d_{b_1,c,n})$ because $\varphi(U^d_{b_1,c,n})$ is well-defined if $c>b_1$.

We observe that \cref{symmetry-U} in dimension $d-1$ implies the following:
    \[\varphi(U^d_{b_1,c,n})=L_c^n(\varphi(U^d_{b_1,c,1}))
    \label{varphi-u-relation}\]
Note that $\varphi(U^d_{b_1,c,1})$ fits inside a box $U^{d-1}_{b_2,c,1}$ for some $b_2\sim b_1$ now that we are at scale $b_1\times\cdots\times b_1$. By \cref{symmetry-U} in dimension $d$ with \cref{varphi-u-relation}, we deduce that the box construction of $U$ nearly commutes with $\varphi$. That is,
    \[\varphi(U^d_{b_1,c,n})\subset U^{d-1}_{b_2,c,n}
    \label{nearly-commutes}\]
Therefore, it suffices to show that $\varphi(\nu^\perp)$ does not meet more than $d-1$ of the boxes $U^{d-1}_{b_2,c,n}$. We would be done by the induction hypothesis if $\varphi(\nu^\perp)$ were to pass through the origin, so we assume that $0\notin\varphi(\nu^\perp)$. Consider boxes $Q(\mathcal M_{d-1}(b_3t_{c,m+1}))$ and take $m_0$ to be maximal so that $\varphi(\nu^\perp)$ contains a point $l$ inside such a box. Here, $b_3\sim b_2$ is so that the boxes $U^{d-1}_{b_2,c,k}$ lie in $Q(\mathcal M_{d-1}(b_3t_{c,m+1}))$ for $k>m+1$. We observe that, for $n>m_0+1$, the hyperplane $\varphi(\nu^\perp)$ does not meet $U^{d-1}_{b_2,c,n}$ by maximality of $m_0$.

If $n=m_0+1$, then it does not matter if $\varphi(\nu^\perp)$ meets $U^{d-1}_{b_2,c,n}$ if we can show that at most $d-2$ of the values $n<m_0$ are bad.

If $n<m_0+1$, we observe that $\varphi(\nu^\perp)=l+\nu_2^\perp$, where $\nu_2$ is perpendicular to $\varphi(\nu^\perp)$. Thus, it suffices to show that no more than $d-2$ of the translated boxes $U^{d-1}_{b_2,c,n}-l$ meet $\nu_2^\perp$. However, since $l\in Q(\mathcal M_{d-1}(b_3t_{c,m+1}))$, we deduce that $U^{d-1}_{b_2,c,n}-l\subset U^{d-1}_{b_4,c,n}$ for some $b_4\sim b_2$. Our inductive hypothesis then implies that at most $d-2$ of the boxes $U^{d-1}_{b_4,c,n}$ meet $\nu_2^\perp$. This completes the proof.
\end{proof}

We are now ready to handle the general case of a curved $C^2$ hypersurface $\srf$. We assume that $\srf$ is given near $\xi=0$ as the graph of a $C^2$ function $\Phi:\R^{d-1}\to\R$. We write
    \[\Phi(\omega)=C(\omega,\omega)+o(|\omega^2|)
    \label{c2-smoothness}\]
where $C$ is some quadratic form. We assume additionally that the largest-magnitude eigenvalue of $C$ corresponds to an eigenvector pointing in the $\eta_1=(1,0,\cdots,0)$ direction. It is also convenient to assume by rescaling that $C(\eta_1,\eta_1)=1$. We will now set
    \[\omega_n=(c\ie n,c\ie{3n},c\ie{4n},\cdots,c\ie{dn})
    \label{def-omega-n}\]
for some $c>1$. We set $\xi_n=\Phi(\omega_n)$.

We are now ready to prove \cref{lemma-pigeonholing} for $\srf$.

\begin{lemma}\label{lemma-srf-sufficient}
    For sufficiently large $n>0$, each $\xi_n$ as defined above lies in the box $U^d_{b,c,n}$ for some choice of $b,c$ consistent with \cref{lemma-moment-curve}.
\end{lemma}
\begin{proof}
Suppose we set $\tilde{\mathcal M}_d(t)=\Phi(t,t^3,t^4,\cdots,t^d)$ so that $\xi_n=\tilde{\mathcal M}_d(c\ie n)$. Note that $\tilde{\mathcal M}_d(t)-\mathcal M_d(t)$ lies entirely in the $(0,1,\cdots,0)$-direction, so it suffices to prove that this difference lies within the $(0,1,\cdots,0)$-width of $U^d_{b,c,n}$. Recall that this width is $O(|t^2|)$, so it suffices to show that
    \[\tilde{\mathcal M}_d(t)-\mathcal M_d(t)=o(|t^2|)\]
However, we have
    \[B((t,0,\cdots,0),(t,0,\cdots,0))=t^2\]
and
    \[|(0,t^3,t^4,\cdots,t^d)|=O(|t|^3)\]
Since the derivative of $\omega\to B(\omega,\omega)$ tends to zero as $\omega\to0$, we conclude that  
    \[B((t,t^3,t^4,\cdots,t^d),(t,t^3,t^4,\cdots,t^d))=B((t,0,\cdots,0),(t,0,\cdots,0))+o(|t|^3)=t^2+o(|t|^2)\]
by writing $(t,t^3,t^4,\cdots,t^d)=(t,0,\cdots,0)+(0,t^3,t^4,\cdots,t^d)$.

This completes the proof upon applying \cref{c2-smoothness}.
\end{proof}

To finish the proof of \cref{lemma-pigeonholing}, note that we can relabel the $\xi_n$ to assume that $n$ is sufficiently large. The pigeonholing condition \cref{pigeons!} is then satisfied by choosing $c$ sufficiently large. The only step left is then proving that $\scale\sim R\inv$. The smallest separation between two values of $\xi_n$ will occur when $n$ reaches its maximum, i.e. $n\sim N$. In particular, the minimum separation is $\gtrsim |t_N|\sim c\ie N$. We can ensure that $c\ie N\sim R\inv$ by choosing the right constant in $N\sim\log R$.
\begin{remark}
There is actually another proof of \cref{lemma-pigeonholing} that relies on a polynomial-partitioning technique. Instead of inducting on the dimension $d$, one can project the moment curve onto a line to obtain a polynomial whose coefficients are suitably bounded on both sides. The problem then reduces to a problem about the shape of sets where the polynomial and its derivatives have certain signs. However, this polynomial-based approach requires some delicate estimates on the shapes of these sets and how these shapes interact with the corresponding shapes for its derivative. It is also very difficult to yield the sharp estimate of $d-1$ for the number of bad points using this approach, and it is unclear how to construct the sharp examples for $\nu$. It is for these reasons that the approach included is preferred.
\end{remark}
\subsection{The proof of Lemma \ref{main-lemma} using Lemma \ref{lemma-pigeonholing}}\label{subsec-prove-lemma}

To prove the plane condition \cref{plane-infty-main-lemma}, it suffices to show that the projections $\pi_\nu\p{\vec c\cdot\vec\xi}$ are distinct at scale $R\inv$. Note that, even though \cref{lemma-pigeonholing} gave only distinct values of $\log_2\xi_i$, we can easily guarantee $\log_K\xi_i$ for any $K>0$ using only every $\log_2(K)$th value of $\xi_i$. Let us fix $\vec c'\in\{0,1\}^N$. If
    \[\abs{\pi_\nu\p{\vec c\cdot\vec\xi}-\pi_\nu\p{\vec c'\cdot\vec\xi}}\leq R\inv
        \label{the-bad-c}\]
holds for more than $2^{d-1}$ values of $\vec c$, then \cref{the-bad-c} holds for at least one $\vec c'$ so that $c'_m\neq c_m$ for some $m\notin S_\nu$, that is, $m$ is not a bad value, and $c'_n=c_n$ for all $n\in S_\nu$. Let us assume that $m$ is the value in $\{1,\cdots,N\}$ where $\pi_\nu\p{\xi_m-\xi_0}$ is maximized. Next, note that since the values of $\pi_\nu\p{\xi_k-0}$ lie in distinct logarithmic intervals of the form $[\alpha,\alpha K]$, we deduce that $|\pi_\nu\p{\xi_m-\xi_0}|$ is more than twice the sum of all $|\pi_\nu\p{\xi_k-\xi_0}|$ for $k$ with $c_k\neq c'_k$. This contradicts \cref{the-bad-c} by the triangle inequality applied to $\sum_l\pm\pi_\nu\p{\xi_l-\xi_0}$, where the sum is over all values of $l$ where $c_l\neq c'_l$.

\printbibliography
\end{document}